\documentclass[10pt,reqno]{amsart}
\usepackage{amssymb,amsfonts,amsthm,amscd,stmaryrd,dsfont,esint,upgreek,constants,todonotes}
\usepackage{amsmath,amsthm,amssymb,amsfonts}
\usepackage[shortlabels]{enumitem}
\usepackage{mathtools, mathrsfs, xargs}
\usepackage{hyperref}
\hypersetup{colorlinks, citecolor=violet, linkcolor=blue, urlcolor=blue}
\newif\iflabels
\labelsfalse 
\iflabels 
  \usepackage[left=0.5in, right=2.5in, marginparwidth=2.2in]{geometry}
  \usepackage{showlabels}
\else
  \usepackage[left=1.25in, right=1.25in]{geometry}
\fi

\newcommand{\be}{\begin{equation}}
\newcommand{\ee}{\end{equation}}

\numberwithin{equation}{section}
\theoremstyle{plain}                
\newtheorem{theorem}{Theorem}[section]
\newtheorem{lemma}[theorem]{Lemma}
\newtheorem{proposition}[theorem]{Proposition}
\newtheorem{corollary}[theorem]{Corollary}

\theoremstyle{definition}           
\newtheorem{definition}[theorem]{Definition}

\newtheorem{assumption}[theorem]{Assumption}

\theoremstyle{remark}
\newtheorem{remark}[theorem]{Remark}

\providecommand{\alias}{}
\renewcommand{\alias}[1]{\providecommand{#1}{}\renewcommand{#1}}

  \DeclarePairedDelimiter\ab{\langle}{\rangle} 
  \DeclarePairedDelimiter\abs{\lvert}{\rvert}   
  \DeclarePairedDelimiter\norm{\lVert}{\rVert}  

  \DeclarePairedDelimiterX\set[1]\{\}{ #1 }
  \DeclarePairedDelimiterX\sets[2]\{\}{ #1\,:\,#2 }

  \let\bPeexp\exp
  \let\exp\relax
  \DeclarePairedDelimiterXPP\exp[1]{\bPeexp}(){}{#1}

  \makeatletter
    \let\oldabs\abs \def\abs{\@ifstar{\oldabs}{\oldabs*}}
    \let\oldab\ab \def\ab{\@ifstar{\oldab}{\oldab*}}
    \let\oldnorm\norm \def\norm{\@ifstar{\oldnorm}{\oldnorm*}}
    \let\oldexp\exp \def\exp{\@ifstar{\oldexp}{\oldexp*}}
  \makeatother

\makeatletter
  \newcommand{\opnorm}{\@ifstar\@opnorms\@opnorm}
  
  \newcommand{\@opnorm}[2][]{%
    \mathopen{#1|\mkern-1.5mu#1|\mkern-1.5mu#1|}
    #2
    \mathclose{#1|\mkern-1.5mu#1|\mkern-1.5mu#1|}
  }
\makeatother

\alias{\R}{{\mathbb R}}
\alias{\C}{{\mathbb C}}
\alias{\Z}{{\mathbb Z}}
\alias{\N}{{\mathbb N}}

\newcommand{\sA}{\mathcal{A}}

\newcommand{\sS}{\mathcal{S}}

\newcommand{\sF}{\mathcal{F}}

\newcommand{\bP}{\mathbb{P}}

\newcommand{\bE}{\mathbb{E}}
\newcommand{\bQ}{\mathbb{Q}}

\newcommand{\kZ}{\kZ}

\newcommand{\lpee}{L^p}

\newcommand{\linf}{L^{\infty}}

\newcommand{\loi}{L^{1,\infty}}

\newcommand{\sinf}{\sS^{\infty}}

\newcommand{\spee}{\sS^p}

\newcommand{\bmo}{\text{bmo}}

\newcommand{\const}[1]{C=C(#1)}
\newcommand{\leqc}{\leq_C}

\newcommand{\apb}{\mathbf{A}}

\newcommand{\sam}{\set{a_m}}

\newcommand{\calpha}{C^{\alpha}}

\newcommand{\cbeta}{C^{\beta}}

\newcommand{\tr}{\text{tr}}
\newcommand{\all}{[0,T] \times \R^d \times \R^n \times (\R^d)^n}
\newcommand{\cdiag}{C_{\text{diag}}}

\newcommand\blfootnote[1]{%
  \begingroup
  \renewcommand\thefootnote{}\footnote{#1}%
  \addtocounter{footnote}{-1}%
  \endgroup
}

\begin{document}
\title[Quadratic FBSDE systems and stochastic differential games]{Global existence for quadratic FBSDE systems and application to stochastic differential games}
\author{Joe Jackson}
\address{Department of Mathematics, The University of Texas at Austin}
\email{jjackso1@utexas.edu}
\thanks{During the preparation of this work the first author has been supported by the National Science Foundation under Grant No. DGE1610403 (2020-2023). Any opinions, findings and conclusions or recommendations expressed in this material are those of the author(s) and do not necessarily reflect the views of the National Science Foundation (NSF).
}

\begin{abstract}
   In this note, we use Girsanov's Theorem together with results from the quadratic BSDE literature to construct global strong solutions for quadratic FBSDE systems. Then, we identify a general class of stochastic differential games whose corresponding FBSDE systems are covered by our main existence result. This leads to the existence of Markovian Nash equilibria for such games. 
\end{abstract}

\maketitle

\section{Introduction}
\blfootnote{The author wishes to thank Daniel Lacker and Ludovic Tangpi for helpful comments on an early version of this note.}
Recent years have witnessed much activity and progress in the area of quadratic BSDE systems, i.e. systems of backward stochastic differential equations (BSDEs) whose driver $f$ has quadratic growth in the control variable, typically denoted $z$. In the Markovian case, the most general global existence results appear in \cite{xing2018}, while in the non-Markovian case global existence is obtained under various structural conditions in \cite{HuTan16}, \cite{Nam19}, and \cite{jackson2021existence}. Fewer efforts have been made to understand quadratic systems of forward-backward stochastic differential equations (FBSDEs), possibly because existence for general FBSDEs is a very challenging problem even when all coefficients are Lipschitz. The works we are aware of which consider quadratic FBSDE systems are \cite{Antonelli2006ExistenceOT}, \cite{fromm2013existence}, \cite{kupperluo} and \cite{Luo2017SolvabilityOC}, which all require either smallness or some type of monotonicity condition.

In this note, we consider the FBSDE
\begin{align} \label{fbsde}
    \begin{cases} dX_t = b(t,X_t,Y_t,Z_t) dt + \sigma(t,X_t) dB_t, \\
    dY_t = - f(t,X_t,Y_t,Z_t) dt + Z_t dB_t, \,\,
    Y_T = g(X_T).
    \end{cases}
\end{align}
We are particularly interested in the case that $Y$ is multidimensional and the driver $f = f(t,x,y,z)$ exhibits quadratic growth in the variable $z$. In particular, the first objective of this note is to extend the global existence results for quadratic BSDE systems obtained in \cite{xing2018} to the quadratic FBSDE \eqref{fbsde}. The FBSDE \eqref{fbsde} is related, at least formally, to two other analytical objects: the BSDE 
\begin{align} \label{bsde}
    \begin{cases} 
    dX_t = \sigma(t,X_t) dB_t, \\
    dY_t^i = -F^i(t,X_t,Y_t,Z_t) dt + Z_t dB_t, \,\, Y_T = g(X_T)
    \end{cases}
\end{align}
where here an throughout the paper we use the convention
\begin{align} \label{Fdef}
    F^i(t,x,y,z) := \big(f^i(t,x,y,z) + z^i \cdot \sigma^{-1}(t,x)b(t,x,y,z)\big)
\end{align}
and, setting $a = \frac{1}{2} \sigma \sigma^T$, the partial differential equation (PDE) 
\begin{align} \label{pde} 
    \partial_t u^i + \tr(a(t,x) D^2 u^i) + f^i(t,x,u,  Du \sigma ) + D u^i \cdot b(t,x,u, Du \sigma) = 0, \, \, u^i(T,x) = g^i(x).
\end{align}

\subsection{Main results.}
The first contribution of this note is an existence result (Theorem \ref{thm.exist}) for \eqref{fbsde} when $f$ exhibits quadratic growth in $z$ but satisfies the structural conditions \eqref{hyp.ab} and \eqref{hyp.bf}, and the data $\sigma$, $b$, and $g$ satisfy some minimal regularity conditions. The proof relies on a sequence of a-priori estimates. Together with a somewhat standard approximation procedure, these a-priori estimates allow us to produce a solution to \eqref{fbsde} through a compactness argument. The first a-priori estimate is Lemma \ref{lem.ab}, which shows that the structural condition \eqref{hyp.ab} leads to $\linf$ estimates on the decoupling field (see Definition \ref{def.decoup}) of \eqref{fbsde}. We emphasize that \eqref{hyp.ab} is only a convenient condition to guarantee a-priori estimates in $\linf$; if such a-priori estimates are established through another method the rest of the analysis goes through unchanged. The second important estimate is Proposition \ref{prop.holder}, a H\"older estimate which follows more or less directly from a result of \cite{xing2018}, thanks to the fact that if the driver $f$ satisfies the structural condition \eqref{hyp.bf}, then so does the driver $F$ given in \eqref{Fdef}. The final estimate is Proposition \ref{prop.calpha}, which shows that an estimate on the H\"older regularity of a solution to \eqref{pde} yields an interior $C^{1,\alpha}$ estimate. This estimate allows us to construct a Markovian solution to the BSDE \eqref{bsde} which is regular enough to also be a decoupling field for the FBSDE \eqref{fbsde}. We emphasize that we require very little regularity of the driver $f$, to be obtain our estimates and existence result, in partiular $f$ need not be even locally Lipschitz in $(y,z)$.

The second contribution is to apply our results to a class of stochastic differential games. Typically, quadratic BSDE systems arise when stochastic differential games (with uncontrolled drift and quadratic costs) are treated through the popular weak formulation. But if the same games are treated in strong formulation, then a quadratic FBSDE arises in place of the quadratic BSDE - roughly speaking, in order to find a Markovian Nash equilibrium, one must solve \eqref{fbsde} in place of \eqref{bsde}. We emphasize that in this approach the FBSDE involved is not the one obtained from the stochastic maximum principle, but the one which represents the value of the game. We make this connection between Markovian Nash equilibria and FBSDEs precise under fairly general conditions in Proposition \ref{pro:fbsdeclne}. Then, we identify a general class of stochastic differential games whose corresponding FBSDEs have a structure covered by Theorem \ref{thm.exist}. These games are characterized by a diagonal cost structure (player $i$'s control does not enter player $j$'s running cost, when $i \neq j$) and a drift $b = b(t,x,a_1,...,a_n)$ which decomposes additively as $b(t,x,a_1,...,a_n) = \sum_{j = 1}^n b_j(t,x,a_j)$ (see Section 4.1 for notation). This leads to an existence result for Markovian Nash equilibria, which is stated precisely in Proposition \ref{prop:clneexist}.

\subsection{Preliminaries and notations}
The dimensions $n$ and $d$ are fixed throughout the paper, as is the terminal time $T \in (0,\infty)$. We also fix throughout the paper a probability space $(\Omega, \sF, \bP)$ hosting a $d$-dimensional Brownian motion $B$, whose augmented filtration is denoted by $\mathbb{F} = (\sF_t)_{0 \leq t \leq T}$. We use the usual notation $\lpee$, $1 \leq p \leq \infty$ for the space of $p$-integrable $\sF_T$-measurable random variables with norm $\norm{\cdot}_{\lpee}$. For a continuous and adapted process $Y$ taking values in some Euclidean space, we define $\norm{Y}_{\spee} = \norm{\sup_{0 \leq t \leq T} |Y_t|}_{\lpee}$, and we write $\bmo$ for the set of all adapted processes $Z$ such that $\norm{Z}_{\bmo} = \sup_{\tau} \bE_{\tau}[\int_{\tau}^T |Z_s|^2 ds] < \infty$, the supremum being taken over all stopping times $ 0 \leq \tau \leq T$ and $\bE_{\tau}[\cdot]$ denoting condition expectaition with respect to $\sF_{\tau}$. Finally, we mention that all the spaces and norms here can be extended in natural ways to include processes defined only on $[t,T]$, for some $t \in [0,T]$. 

Let us mention that we will write $Dv$ for the spatial gradient of a map $v : [0,T] \times \R^d \to \R$, and for $u = (u^1,...,u^n) : [0,T] \times \R^d \to \R$, $Du$ will denote $(Du^1,...,Du^n)$, viewed as an element of $(\R^d)^n$. We will also view the unknown $Z$ appearing in \eqref{fbsde} (and \eqref{bsde}) as taking values in $(\R^d)^n$. We will manipulate elements of $(\R^d)^n$ in a natural ``element-wise way" as in \cite{jackson2021existence}, e.g. if $p \in (\R^d)^n$ and $q \in \R^{d \times d}$, then $pq$ denotes an element of $(\R^d)^n$ whose $i^{th}$ entry is $q p^i$. Likewise if $p \in (\R^d)^n$ and $q \in \R^d$, then $pq \in \R^n$ and $(pq)^i = p^i \cdot q$. This philosophy will in particular be used when interpreting the stochastic differential $Z_t dB_t$ and expressions like $Z_t \sigma(t,X_t)$.

We will be working with certain parabolic H\"older spaces, defined as follows. Fix $\alpha \in (0,1)$. For a function $v  = v(t,x) : [0,T] \times \R^d \to E$, $E$ being some Euclidean space with norm $| \cdot |$, we define the H\"older seminorm 
\begin{align*}
    [v]_{\calpha} = [v]_{\calpha([0,T] \times \R^d)} = \sup_{t \neq t', x \neq x'} \frac{|v(t,x) - v(t',x')|}{|t - t'|^{\alpha/2} + |x - x'|^{\alpha}}. 
\end{align*}
Next, we define $\norm{u}_{\calpha} = \norm{u}_{\linf} + [u]_{\calpha}$, and $\norm{u}_{C^{1,\alpha}} = \norm{u}_{\calpha} + \norm{Du}_{\calpha}$. Given an open sut $U \subset [0,T] \times \R^d$, we define $\norm{u}_{C^{\alpha}(U)}$ and $\norm{u}_{C^{1 + \alpha}(U)}$ similarly. We define the H\"older spaces of functions defined on $\R^d$ in the same way, i.e. for $g : \R^d \to E$, $
\norm{g}_{\calpha} = \sup_{x \neq x'} \frac{|g(x) - g(x')|}{|x - x'|}$.

At this point, we need to make precise the notions of solutions we will be working with.
\begin{definition} \label{def.decoup}
A pair of measurable functions $u : [0,T] \times \R^d \to \R^n$, $v : [0,T] \times \R^d \to  (\R^d)^n$ with $u$ bounded and continuous is called a \textbf{decoupling field} for \eqref{fbsde} if for each $t \in [0,T]$ and $x \in \R^d$, the SDE 
\begin{align} \label{sde}
    X_{t'}^{t,x} = x + \int_{t}^{t'} b(s,X_s^{t,x}, u(s,X_s^{t,x}), v(s,X_s^{t,x})) ds + \sigma(s,X^{t,x}_s)  dB_s
\end{align}
has a unique strong solution on $[t,T]$, and with $(Y^{t,x},Z^{t,x})$ defined by $(Y^{t,x},Z^{t,x}) = \big(u(\cdot,X^{t,x}), v(\cdot,X^{t,x})\big)$, the triple $(X^{t,x}, Y^{t,x}, Z^{t,x})$ solves the equation 
\begin{align} \label{fbsdedecoup}
   \begin{cases} X_{t'}^{t,x} = x + \int_{t}^{t'} b(s,X_s^{t,x},Y_s^{t,x},Z^{t,x}_s) ds + \int_{t}^{t'} \sigma(s,X_s^{t,x}) dB_s, \\
    Y_{t'}^{t,x} = g(X_T^{t,x}) + \int_{t'}^T f(s,X_s^{t,x},Y_s^{t,x}, Z^{t,x}_s) ds - \int_{t'}^T Z_s^{t,x} dB_s 
    \end{cases}
\end{align}
on the interval $[t,T]$. We call $(u,v)$ a $\bmo$ decoupling field if $Z^{t,x} \in \bmo$, for each $(t,x) \in [0,T] \times \R^d$. 

A pair of measurable functions $u : [0,T] \times \R^d \to \R^n$, $v : [0,T] \times \R^d \to (\R^d)^n$ with $u$ bounded and continuous is called a
\textbf{Markovian solution} to \eqref{bsde} if for each $t \in [0,T]$ and $x \in \R^d$, the pair $(Y^{t,x},Z^{t,x})$ defined by $(Y_{t'}^{t,x},Z_{t'}^{t,x}) \coloneqq \big(u(t',X^{t,x}_{t'}), v(t', X^{t,x}_{t'})\big)$ solves the BSDE
\begin{align} \label{bsdemarkov}
    Y^{t,x}_{t'} = g(X_T^{t,x}) + \int_{t'}^T F(s,X^{t,x}_s,Y^{t,x}_s,Z^{t,x}_s) ds - \int_{t'}^T Z^{t,x}_s dB_s
\end{align}
on the interval $[t,T]$,
where $X^{t,x}$ is the unique strong solution of 
\begin{align*}
    X^{t,x}_{t'} = x + \int_{t}^{t'} \sigma(s,X^{t,x}_s)  dB_s. 
\end{align*}
We call $(u,v)$ a $\bmo$ Markovian solution if $Z^{t,x} \in \bmo$ for each $(t,x) \in [0,T] \times \R^d$.
\end{definition}

\begin{remark}
We note that our definition of decoupling field differs from the usual one in that we include the function $v$ as part of the decoupling field. This is to make the relationship between the decoupling field for \eqref{fbsde} and the Markovian solution of \eqref{bsde} easier to state. Moreover, we note that the existence of a decoupling field for \eqref{fbsde} implies the existence, for any $x \in \R^d$, of a \textit{strong} solution to the equation 
\begin{align}  \label{fbsdestrong}
    \begin{cases}
        X_t = x + \int_0^t b(s,X_s,Y_s,Z_s) ds + \int_0^t \sigma(s,X_s)  dB_s, \\
    Y_t = g(X_T) + \int_t^T f(s,X_s,Y_s,Z_s) ds - \int_t^T Z_s  dB_s,
    \end{cases}
\end{align}
i.e. a pair of adapted processes $(Y,Z)$ satisfying \eqref{fbsdestrong} path-wise a.s. 
\end{remark}

The following is a consequence of It\^o's formula and the Girsanov transform.

\begin{proposition} \label{prop:basic}
Suppose that \ref{hyp.sigmab} and \ref{hyp.q} hold, and that $(u,v)$ is a $\bmo$ Markovian solution to \eqref{bsde} such $v$ is bounded on $[0,t] \times \R^d$ for any $t < T$. Then $(u,v)$ is also a $\bmo$ decoupling field for \eqref{fbsde}. Conversely, any $\bmo$ decoupling field is also a $\bmo$ Markovian solution of \eqref{bsde}.
\end{proposition}

\begin{proof}
Let us first assume that $(u,v)$ is a $\bmo$ decoupling field for \eqref{fbsde}. For fixed $t$ and $x$, let $X^{t,x}$ be defined by \eqref{sde} and set $(Y^{t,x}, Z^{t,x}) = \big(u(\cdot, X^{t,x}), v(\cdot, X^{t,x})\big)$. By the definition of decoupling field, we have the relationship
\begin{align*}
    Y_{t'}^{t,x} &= g(X_T^{t,x}) + \int_{t'}^T f(s,X_s^{t,x}, Y^{t,x}_s, Z_s^{t,x}) ds - \int_{t'}^T Z^{t,x}_s  dB_s \\
    &= g(X_T^{t,x}) + \int_{t'}^T  F(s,X_s^{t,x}, Y^{t,x}_s, Z_s^{t,x}) ds- \int_{t'}^T Z^{t,x}_s
     \big(\sigma^{-1}(s,X^{t,x}_s)  dX_s^{t,x}\big). 
\end{align*}
Recalling the definitions of $Y^{t,x}$ and $Z^{t,x}$, we find that 
\begin{align*}
    u(t',X^{t,x}_{t'}) = g(X_T^{t,x}) + \int_{t'}^T  F(s,X_s^{t,x}, u(s,X^{t,x}_s), v(s,X^{t,x}_s)) ds- \int_{t'}^T v(s,X^{t,x}_s)  \big(
    \sigma^{-1}(s,X^{t,x}_s)  dX_s^{t,x}\big).
\end{align*}
Now, if $\tilde{X}^{t,x}$ denotes the solution of 
$
    \tilde{X}^{t,x}_{t'} = x + \int_t^{t'} \sigma(s,\tilde{X}_s^{t,x})  dB_s, 
$
then since $Z^{t,x} \in \bmo$, and $|b(t,x,y,z)| \leq C_0(1 + |z|)$, Girsanov's Theorem yields a probability measure $\bQ$ such that the law of $\tilde{X}^{t,x}$ under $\bQ$ is the same as the law of $X^{t,x}$ under $\bP$. Thus the relationship
\begin{align*}
    u(t',\tilde{X}^{t,x}_{t'}) &= g(\tilde{X}_T^{t,x}) + \int_{t'}^T  F(s,\tilde{X}_s^{t,x}, u(s,\tilde{X}^{t,x}_s), v(s,\tilde{X}^{t,x}_s)) ds- \int_{t'}^T v(s,\tilde{X}^{t,x}_s)
    \big( \sigma^{-1}(s,\tilde{X}^{t,x}_s) d\tilde{X}_s^{t,x}\big) \\
    &= g(\tilde{X}_T^{t,x}) + \int_{t'}^T  F(s,\tilde{X}_s^{t,x}, u(s,\tilde{X}^{t,x}_s), v(s,\tilde{X}^{t,x}_s)) ds- \int_{t'}^T v(s,\tilde{X}^{t,x}_s)
     dB_s
\end{align*}
holds under the measure $\bQ$, hence also under $\bP$. Thus $(u,v)$ is a Markovian solution to \eqref{bsde}. 

Now suppose that $(u,v)$ is a $\bmo$ Markovian solution to \eqref{bsde}, with $v$ bounded on $[0,T-\epsilon]$ for any $\epsilon > 0$. Then for any $(t,x)$ and $\epsilon > 0$, the SDE \eqref{sde} has a unique strong solution on $[t,T)$, thanks to a classical result which can be traced to Veretennikov (see \cite{ZHANG20051805} and the references therein for more information about the solvability of SDEs with irregular drift). The fact that $v$ is a $\bmo$-decoupling field implies a bound on the process $
    \tilde{b}_s \coloneqq b(s,X^{t,x}_s, u(s,X^{t,x}_s), v(s,X^{t,x}_s))$
in, say, $L^2(\Omega \times [t,T))$. Together with the boundedness of $\sigma$, this implies easily that a.s., $X^{t,x}_s$ has a limit as $s \to \infty$, which lets us extend $X^{t,x}$ uniquely to all of $[t,T]$. Now we set $(Y^{t,x},Z^{t,x}) = (u(\cdot,X^{t,x}),v(\cdot, X^{t,x}))$. Checking that $(Y^{t,x}, Z^{t,x})$ solves \eqref{fbsdedecoup} amounts to running the above change-of-measure argument in reverse. 
\end{proof}

\subsection{Assumptions}
We now describe some conditions on the data, which consists of measurable maps
\begin{align*}
    &b = b(t,x,y,z) : \all \to \R^d, \quad \sigma = \sigma(t,x) : [0,T] \times \R^d \to \R^{d \times d}, \\
    &f = f(t,x,y,z) : \all \to \R^n, \quad g = g(x) : \R^d \to \R^n
\end{align*}
which we will later impose in various combinations in order to get estimates and existence results. We start with the conditions on $\sigma$ and $b$ which will be used throughout the paper.
\be \label{hyp.sigmab} \tag{$H_{0}$}
\begin{cases}
\text{There exists a constant $C_{0}$ such that $\sigma$ and $b$ satisfy } \\ 
\hspace{.5cm} 1) \,\, \frac{1}{C_{0}} |w|^2 \leq |\sigma(t,x) w|^2 \leq C_{0} |w|^2, \\
\hspace{.5cm} 2) \,\, |\sigma(t,x) - \sigma(t,x')| \leq C_{0} |x - x'|, \\
\hspace{.5cm} 3) \,\, |b(t,x,y,z)| \leq C_0 (1 + |z|) \\
\text{for all } t \in [0,T], x,x',w \in \R^d, z \in (\R^d)^n
\end{cases}
\ee
The next condition will be used to guarantee an a-priori estimate on $\norm{Y}_{\sinf}$ for the equation \eqref{fbsde}, provided that the terminal condition is bounded (see Lemma \ref{lem.ab}). 
\be \label{hyp.ab} \tag{$H_{AB}$} \begin{cases}
\text{There exists a constant $\rho$ and a finite collection $\sam = (a_1,\dots, a_M)$}  \\ \text{of
  vectors in $\R^n$ such that $a_1,\dots, a_M$ positively span $\R^n$, and } \\ a_m^T f(t,x,y,z) \leq \rho + \tfrac{1}{2} \abs{a_m^T z}^2 \text{ 
    for each $m$, for all $(t,x,y,z) \in \all$.}
\end{cases}
\ee 
The next condition states that the driver $f$ has quadratic growth in $z$. 
\be \label{hyp.q} \tag{$H_Q$}
\begin{cases}
\text{There exists a constant $C_Q$ such that the estimate} \\
\hspace{.5cm} |f^i(t,x,y,z)| \leq C_Q(1 + |z|^2) \\
\text{holds for each } (t,x,y,z) \in \all, \,\, i = 1,...,n.
\end{cases}
\ee
It is well-known that a quadratic growth assumption like \ref{hyp.q} is not enough to obtain regularity estimates on the PDE system \eqref{pde}, so we will impose the following structural condition. The condition can be traced back to \cite{bengame}, and a similar condition appeared in \cite{xing2018}, where it was termed the Bensoussan-Frehse condition. 
\be \label{hyp.bf} \tag{$H_{BF}$}
\begin{cases}
\text{There exists a constant $C_Q$ and a sub-quadratic function $\kappa : \R_+ \to \R_+$ such that } \\
\hspace{.5cm} |f^i(t,x,y,z)| \leq C_Q\big(1 + |z^i||z| + \sum_{j < i} |z^j|^2 + \kappa(|p|) \big) \\
\text{for all } (t,x,y,z) \in \all, \,\, i = 1,...,n. 
\end{cases}
\ee
\section{A-priori estimates and existence}

\begin{lemma} \label{lem.ab}
Let $(u,v)$ be a Markovian solution of \eqref{bsde}, and suppose that \ref{hyp.ab} holds. Then we have 
\begin{align*}
\norm{u}_{\linf} \leq C, \,\, \const{\norm{g}_{\linf}, \rho, \sam}. 
\end{align*}
\end{lemma}

\begin{proof}
Dropping the superscripts, we define $X$ to be the solution of the equation 
$
    X_t = \int_0^t \sigma(s,X_s) dB_s. 
$
We then set $(Y,Z) = (u(\cdot, X), v(\cdot, X))$. Since $(u,v)$ is a Markovian solution to \eqref{bsde}, we have 
\begin{align*}
    dY_t^i =  - f^i(t,X_t,Y_t,Z_t)dt  + Z_t^i \cdot d\tilde{B}_t, 
\end{align*}
where $\tilde{B} = B - \int \sigma^{-1}(\cdot,X)b(\cdot,X, Y, Z) dt$. 
Since $\sigma^{-1}$ is bounded, $|b(t,x,y,z)| \leq C_0(1 + |z|)$, and $Z \in \bmo$, we deduce that $\tilde{B}$ is a Brownian motion under the measure
$\tilde{\bP}$, where 
$
    d\tilde{\bP} = \mathcal{E}\big(\int \sigma^{-1}(\cdot,X)b(\cdot, X,Y,Z) \cdot dB \big) d\bP. 
$
Now we consider the process
$
    R_t \coloneqq \exp{2a_m^T Y +  \int_0^{\cdot} 2 \rho_t dt}. 
$
We compute 
\begin{align*}
    dR_t = R_t \big(- 2 a_m^T f + 2 |a_m^TZ_t|^2 + 2\rho_t)dt + 2R_t a_m^T Z_t d\tilde{B}_t. 
\end{align*}
Since $f \in \apb(\rho,\sam)$, $a_m^T f(t,x,y,z) \leq \rho + |a_m^Tz|^2$, and so $
    -a_m^T f(t,x,y,z) + |a_m^T z|^2 + \rho \geq 0. $
In particular, $R$ is a submartingale with terminal element $R_T = \exp{2a_m^T g(X_T) + 2 \int_0^T \rho_s ds}$, which satisfies $\norm{R_T}_{\linf} \leq C$, $\const{a_m, \norm{g}_{\linf}, \norm{\rho}_{\loi}}$. From the definition of $R$, we see that for each $m$ we have
\begin{align*}
    \sup_{0 \leq t \leq T} a_m^T Y_t \leq C, \, \, \const{\norm{g}_{\linf}, \norm{\rho}_{\loi}, \sam}. 
\end{align*}
Since $\sam$ positively spans $\R^n$, 
this gives us an estimate $\norm{Y}_{\sinf}$, which transfers to the desired estimate on $\norm{u}_{\linf}$.
\end{proof}

The following is a consequence of Theorem 2.5 of \cite{xing2018}.

\begin{proposition} \label{prop.holder}
Suppose that \eqref{hyp.bf} holds, $g \in \calpha$ for some $\alpha \in (0,1)$, and that $(u,v)$ is a Markovian solution of \eqref{bsde}, with $u$ bounded. Then for some $\beta \in (0,1)$ depending on $\alpha, \norm{g}_{\calpha}, C_0, \text{ and } \norm{u}_{\linf}$, we have 
\begin{align}
    \norm{u}_{\cbeta} \leq C, \,\, \const{\norm{g}_{\calpha}, C_0, \norm{u}_{\linf}}.
\end{align}
\end{proposition}

\begin{proof}
The only thing to check is that if \eqref{hyp.bf} holds, then $F$ has a decomposition as in (2.8) of \cite{xing2018}, so that Proposition 2.11 of \cite{xing2018} implies the existence of an appropriate Lyapunov function. For this, we set 
\begin{align} \label{bfequiv}
    &l^i(t,x,y,z) = \frac{f^i(t,x,y,z)}{\big(1 + |z^i| |z| + \sum_{j < i} |z^j|^2 + \kappa(|z|)\big)} \frac{ z^i |z|}{|z^i|} 1_{|z^i| \neq 0} + \sigma^{-1}(t,x) b(t,x,y,z), \nonumber \\
    &q^i(t,x,y,z) =\frac{f^i(t,x,y,z)}{\big(1 + |z^i| |z| + \sum_{j < i} |z^j|^2 + \kappa(|z|)\big)} (1 + \sum_{j < i} |z^j|^2), \nonumber, \\
    &s^i(t,x,y,z) = \frac{f^i(t,x,y,z)}{\big(1 + |z^i| |z| + \sum_{j < i} |z^j|^2 + \kappa(|z|)\big)} \kappa(|z|)
\end{align}
Then some algebra shows that we have $F^i(t,x,y,z) = z^i \cdot l^i(t,x,y,z) + q^i(t,x,y,z) + s^i(t,x,y,z)$, and $l$, $q$ and $s$ satisfy the estimates appearing in Proposition 2.11 in \cite{xing2018}. Thus we can apply Theorem 2.5 of \cite{xing2018} to complete the proof.
\end{proof}

\begin{proposition} \label{prop.calpha}
Suppose that \eqref{hyp.q} holds and that $u$ is a classical solution of \eqref{pde} such that $u \in C^{\alpha}$ for some $\alpha \in (0,1)$, and $Du$ is bounded. Then for each $\beta \in (0,1)$ there is a constant $C$ depending on $\beta$, $\alpha$, $\norm{u}_{\calpha}$, $C_0$ and $C_Q$ such that
\begin{align} \label{conealpha}
    \norm{u}_{C^{1 + \beta}([0,t_0] \times \R^d)} \leq \frac{C}{T-t_0}, \quad t_0 \in (0,T), 
\end{align}
Moreover, if $g$ is Lipschitz with Lipschitz constant $L$, then 
\begin{align} \label{lip}
    \norm{Du}_{\linf([0,T] \times \R^d)} \leq C, \const{\alpha, \norm{u}_{\calpha}, \,\,C_0, C_Q,L}.
\end{align}
\end{proposition}

\begin{proof}
Fix $p \in (1,\infty)$. Throughout this proof, $C$ denotes a constant which can change from line to line but depends only on $p$, $\alpha$, $\norm{u}_{\calpha}$, and $C_Q$. We will introduce below parameters $R > 0$ and $t_0 \in [0,T)$, and it is important that $C$ does not depend on $R$ or $t_0$. For constants which can depend on $R$ (but not $t_0$) in addition to the constants $p$, $\alpha$, $\calpha$ and $C_Q$ we use $C_R$. 

We now fix a function $\rho \in C_c^{\infty}(\R^d)$ such that $0 \leq \rho \leq 1$, $\rho(x) = 1$ for $|x| \leq 1$, $\rho(x) = 0$ for $|x| > 2$. Then we define for each $x_0 \in \R^d$ and $R > 0$, the function $\rho^{R,x_0}(x) = \rho(\frac{x - x_0}{R})$, 
and note that $\rho^{R, x_0}(x) = 1$ for $x \in B_R(x_0)$ and $\rho^{R,x_0}(x) = 0$ for $x \in B_{2R}(x_0)^c$. Next, we fix a smooth function $\kappa = \kappa(t) : [0,T] \to [0,1]$ with $\kappa(t) = 1$ for $0 \leq t \leq t_0$ and $\kappa(t) = 0$ for $t > (t_0 + T)/2$. We can choose $\kappa$ so that $|\kappa'(t)| \leq \frac{3}{T - t_0}$. Next, we find the equation satisfied by $\tilde{u}^i(t,x) = \kappa(t) \rho^{R,x_0}(x)u^i(t,x)$. Some computations show that 
\begin{align} \label{tildecomp}
    \partial_t \tilde{u}^i + \tr(a D^2 \tilde{u}^i) + \kappa(t) \rho^{R, x_0}(x) F^i(t,x,u,\sigma Du) 
    = \kappa'(t) \rho^{R,x_0}(x) u^i(t,x) \nonumber \\ + \sum_{j,k} a^{jk} \big( \kappa D_k \rho^{R,x_0} D_ju^i + \kappa D_{kj} \rho^{R,x_0} u^i + \kappa D_j \rho D_k u^i \big).
\end{align}
We use Young's inequality to estimate the right-hand side of \eqref{tildecomp}, and then deduce from the theory of linear parabolic equations the existence of constants $C$ and $C_R$ such that 
\begin{align} 
    \int_0^T \int_{B_{2r}(x_0)}  \big(|\partial_t \tilde{u}^i(t,x)|^p + |D \tilde{u}^i(t,x)|^p + |D^2 \tilde{u}^i(t,x)|^p \big) dx dt \leq C \int_0^T \int_{B_{2R}} \big(|\kappa(t)| |Du|^2 + \frac{C_R}{T-t_0}\big)^p dx dt \nonumber \\
    \leq 2^p C \int_0^T \int_{B_{2R}} |\kappa(t)|^p |Du|^{2p} dx dt + 2^p C \big(\frac{C_R}{T-t_0}\big)^p = C \int_0^T \int_{B_{2R}} |\kappa(t)|^p |Du|^{2p} dx dt + \frac{C_R}{|T-t_0|^p}
\end{align}
holds for all $R \leq 1$ and all $x_0 \in \R^d$, $t_0 \in [0,T)$, and where in the last line we increased $C$ and $C_R$ (and we recall that $C$ and $C_R$ may depend on $p$). Since $\partial_t \tilde{u}^i = \kappa'(t) u^i + \kappa(t) \partial_t u^i$ and $D^k \tilde{u}^i = \kappa(t) D^ku$ on $[0,T] \times B_R(x_0)$, we can infer
\begin{align} \label{est1}
    \int_0^T \int_{B_R(x_0)} |\kappa(t)|^p \big(|\partial_t u|^p + |Du|^p + |D^2 u|^p \big) dx dt &\leq C \int_0^T \int_{B_{2R}} |\kappa(t)|^p |Du|^{2p} dx dt + \frac{C_R}{|T-t_0|^p}
\end{align}
The next step is to set 
\begin{align*}
    c^{R,x_0,i}(t) = \frac{1}{2} \big(\max_{Q_{4R}(x_0)} u^i + \min_{Q_{4R(x_0)}} u^i \big)
\end{align*}
and then follow a computation from \cite{bengame}, integrating by parts in space to find
\begin{align*}
    \int_{0}^T \int_{\R^d} |\kappa(t)|^{p} |\rho^{2R,x_0} Du|^{2p} dx dt = 
    - \int_0^T \int_{\R^d} |\kappa(t)|^p |\rho^{2R, x_0}|^{2p} |Du|^{2p - 2} \sum_i \Delta u^i  (u^i - c^{R,x_0,i}(t)) dx dt \\
    - 2p \int_0^T \int_{\R^d}  |\kappa(t)|^p |\rho^{2R, x_0}|^{2p-1}  \sum_{i,j} D_j \rho^{2R,x_0} |Du|^{2p -2} D_j u^i (u^i - c^{R,x_0,i}(t)) dx dt \\
    - 2(p-1) \int_0^T \int_{\R^d} |\kappa(t)|^p |\rho^{2R,x_0}|^{2p} |Du|^{2p-4} \sum_{i,l,j,k} D_{jk}u^i D_j u^i D_k u^l (u^l - c^{R,x_0,l}(t)) dx dt.
\end{align*}
Applying Young's inequality to the right hand side of the the estimate above, we get
\begin{align*}
    \int_{0}^T \int_{\R^d} |\kappa(t)|^{p} |\rho^{2R,x_0} Du|^{2p} dx dt 
    \leq C \bigg(  \int_{0}^T \int_{\R^d} |\kappa(t)|^{p} |\rho^{2R,x_0}|^{2p} |D^2 u|^p |u - c^{R,x_0}| dx dt  \\
    + \int_{0}^T \int_{\R^d} |\kappa(t)|^{p} |\rho^{2R,x_0}|^{2p} |D u|^{2p} |u - c^{R,x_0}| dx dt + \int_{0}^T \int_{\R^d} |\kappa(t)|^{p} |D\rho^{2R,x_0}|^{2p}|u - c^{R,x_0}| dx dt \bigg)
    \\
    \leq C R^{\alpha} \big( \int_{0}^T \int_{\R^d} |\kappa(t)|^{p} |\rho^{2R,x_0}|^{2p}| |D^2 u|^p dx dt + \int_{0}^T \int_{\R^d} |\kappa(t)|^{p} |\rho^{2R,x_0}|^{2p}| |D u|^{2p} dx dt \big) + C_R
\end{align*}
and so in particular
\begin{align*}
\int_{0}^T \int_{B_{2R}(x_0)} |\kappa(t)|^{p} |Du|^{2p} dx dt \leq C R^{\alpha} \bigg( \int_{0}^T \int_{B_{4R}(x_0)} |\kappa(t)|^{p} |D^2 u|^p dx dt + \int_{0}^T \int_{B_{4R}(x_0)} |\kappa(t)|^{p} |D u|^{2p} dx dt \bigg) + C_R. 
\end{align*}
We can combine this with \eqref{est1} to find that 
\begin{align}
  \sup_{x_0} \int_0^T& \int_{B_R(x_0)} |\kappa(t)|^p \big(|\partial_t u|^p + |Du|^p + |D^2 u|^p \big) dx dt\nonumber \\ &\leq C R^{\alpha} \bigg(  \sup_{x_0} \int_{0}^T \int_{B_{4R}(x_0)} |\kappa(t)|^{p} |D^2 u|^p dx dt  +  \sup_{x_0} \int_{0}^T \int_{B_{4R}(x_0)} |\kappa(t)|^{p} |D u|^{2p} dx dt \bigg) + \frac{C_R}{|T-t_0|^p} \nonumber \\
 &\leq C R^{\alpha} \bigg(  \sup_{x_0} \int_{0}^T \int_{B_{R}(x_0)} |\kappa(t)|^{p} |D^2 u|^p dx dt +  \sup_{x_0} \int_{0}^T \int_{B_{R}(x_0)} |\kappa(t)|^{p} |D u|^{2p} dx dt \bigg) + \frac{C_R}{|T-t_0|^p} 
\end{align}
and so taking $R$ sufficiently small, we conclude 
\begin{align} \label{est2}
    \sup_{x_0} \int_0^T& \int_{B_R(x_0)} |\kappa(t)|^p \big(|\partial_t u|^p + |Du|^p + |D^2 u|^p \big) dx dt \leq \frac{C_R}{|T-t_0|^p}, 
\end{align}
The estimate \eqref{conealpha} now follows from the Sobolev embedding. The proof that \eqref{lip} holds when $g$ is Lipschitz is entirely similar, so we provide only a brief description of the argument. First, we set $v^i$ to be the unique solution to the linear equation $
    \partial_t v^i + \tr(a(t,x) D^2 v^i) = 0$, with terminal condition $v^i(T,x) = g^i(x)$.
Then it is standard that $v^i$ is smooth on $[0,T) \times \R^d$, with $\norm{Dv}_{\linf([0,T] \times \R^d)} \leq C$, $C = C(C_0, L)$. Moreover, $\tilde{u} := u - v$ satisfies \begin{align*}
        \partial_t \tilde{u}^i + \tr(aD^2 \tilde{u}^i) + \tilde{F}^i(t,x,\tilde{u}, \sigma D \tilde{u}) = 0, \,\,
        u(T,x) = 0, 
\end{align*}
where $\tilde{F}^i(t,x,y,z) = F^i(t,x,y + v(t,x), z + \sigma Dv(t,x))$ satisfies \eqref{hyp.q} with a new constant $C_Q'$ depending on $C_Q$ and $\norm{v}_{\linf}$, $\norm{Dv}_{\linf}$. Now we can repeat the same computations as above, but without multiplying by $\kappa$, to get an estimate on $\norm{D\tilde{u}}_{\cbeta([0,T] \times \R^d)}$, which implies the estimate \eqref{lip}. 
\end{proof}

\begin{corollary} \label{cor.interp}
Under the same hypotheses as Proposition \ref{prop.calpha}, for each $\epsilon > 0$ there is a constant $C$ depending on $\epsilon$, $\beta$, $\alpha$, $\norm{u}_{\calpha}$, $C_0$ and $C_Q$ such that
\begin{align*}
\norm{u(t,\cdot)}_{C^{1+\beta}(\R^d)} \leq \frac{C}{(T-t)^{(1 + \beta)/2 + \epsilon}}. 
\end{align*}
In particular, we have for each $\epsilon > 0$ a constant $C$ such that 
\begin{align*}
|Du(t,x)| \leq \frac{C}{(T-t)^{1/2 + \epsilon}}. 
\end{align*}
\end{corollary}

\begin{proof}
Combine Proposition \ref{prop.calpha} with Exercise 3.2.6 of \cite{krylovholder}.
\end{proof}

Now we come to the main existence result.

\begin{theorem} \label{thm.exist}
Suppose that $f$, $b$, and $\sigma$ are continuous and in addition \ref{hyp.sigmab}, \ref{hyp.bf}, and \ref{hyp.ab} hold. Suppose further that $g \in \cbeta$ for some $\beta \in (0,1)$. Then there exists a $\bmo$ decoupling field $(u,v)$ for \eqref{fbsde} such that for some $\alpha \in (0,1)$ $u \in \calpha([0,T] \times \R^d)$ and for each $\epsilon > 0$ there is a $C > 0$ such that $v$ satisfies 
\begin{align} \label{vest}
    |v(t,x)| \leq \frac{C}{(T-t_0)^{1/2 + \epsilon}}
\end{align}
Moreover, $Du \in \calpha([0,t] \times \R^d)$ for each $t < T$, and $v = \sigma Du$. Finally, if $g$ is Lipschitz, then $Du$ is bounded on $[0,T] \times \R^d$.
\end{theorem}

\begin{proof}
First, we truncate in $z$ - in particular, we define $\pi^k : (\R^d)^n \to (\R^d)^n$ by $\pi^k(z) = z$ for $|z| \leq k$, $\pi^k(z) = \frac{k z}{|z|}$ for $|z| > k$. Then we set $f^{(k),i}(t,x,y,z) = f^i(t,x,y,\pi^k(z))$, $b^{(k)}(t,x,y,z) = b(t,x,y,\pi^k(z))$. Then for $\epsilon > 0$, we define $f^{(k),\epsilon,i}$ and $b^{(k),\epsilon}$ through mollification in the variables $(t,x,y,z)$. More precisely, we let $(\rho_{\epsilon})_{0 \leq \epsilon \leq 1}$ be a standard mollifier on $\R \times \R^d \times \R^n \times (\R^d)^n$ and we set
\begin{align*}
    f^{(k),\epsilon,i}(t,x,y,z) = \int_{\R \times \R^d \times \R^n \times (\R^d)^n} f^{(k),i}(t',x',y',z') \rho_{\epsilon}(t - t', x - x', y - y', z - z') dt' dx' dy' dz',
\end{align*}
where we have extended $f^{(k),i}$ to all of $\R \times \R^d \times \R^{n} \times (\R^d)^n$ by $l^{(k),i}(t,x,y,z) = l^{k,i}((0 \vee t) \wedge T, x,y,z)$. We define $b^{(k),\epsilon}$ similarly. Finally, set $g^{\epsilon}$ to be a standard mollification of $g$. Since $b^{\epsilon}$, $f^{(k),\epsilon}$, $g^{\epsilon}$ are all smooth with bounded derivatives of all orders, there is a unique classical solution $u^{(k),\epsilon}$ to the PDE 
\begin{align*}
    \partial_t u^{(k),\epsilon,i} + \frac{1}{2}\tr( a D^2 u^{(k),\epsilon,i} + f^{(k),\epsilon,i}(t,x,u^{(k),\epsilon}, Du^{(k),\epsilon}) + Du^{(k),\epsilon,i} \cdot b^{(k),i}(t,x,u^{(k),\epsilon}, Du^{k,\epsilon}) = 0.
    \end{align*}
    Some computations show that the data $(b^{(k),\epsilon}, f^{(k),\epsilon}, g^{\epsilon})$ satisfy the conditions \eqref{hyp.sigmab}, \eqref{hyp.ab}, and \eqref{hyp.bf} uniformly in the parameters $k$ and $\epsilon$. 
    Applying Propositions \ref{prop.holder} and \ref{prop.calpha} we obtain a constant $C > 0$ such that the estimates $\norm{u^{(k),\epsilon}}_{\calpha} \leq C$, $\norm{Du^{(k),\epsilon}}_{C^{\alpha}([0,t_0] \times \R^d)} \leq \frac{C}{T - t_0}$ hold for each $t_0 < T$ and each $k,\epsilon$. A standard compactness argument gives us a function $u \in C^{\alpha}([0,T] \times \R^d) \cap C^{1+\alpha}_{\text{loc}}([0,T) \times \R^d)$ satisfying the same estimates as the $u^{(k),\epsilon}$, and such that for some $k_j \uparrow \infty$, $\epsilon_j \downarrow 0$, we have $u^{(k_j),\epsilon_j} \to u$ locally uniformly on $[0,T] \times \R^d$ and $ Du^{(k_j),\epsilon_j} \to Du$ locally uniformly on $[0,T) \times \R^d$. Fix $(t,x) \in [0,T] \times \R^d$ and define $X = X^{t,x}$ by \eqref{sde}. By passing to the limit in the equation
    \begin{align*}
        u^{(k_j),\epsilon_j}(t',X_{t'}) &=  u^{(k_j),\epsilon_j}(T,X_T) + \int_{t'}^T F^{(k_j),\epsilon_j}(s,X_s, u^{(k_j),\epsilon_j}(s,X_s), \sigma Du^{(k_j),\epsilon_j}(s,X_s)) ds \\ &- \int_{t'}^T \sigma Du^{(k_j),\epsilon_j}(s,X_s) dB_s
    \end{align*}
    we confirm that the pair $(u,\sigma Du)$ is a Markovian solution for \eqref{bsde}. The boundedness of $u$ and the fact that $F$ admits a Lyapunov function can be used to verify that $(u,\sigma Du)$ is a $\bmo$ decoupling field, and hence by Proposition \ref{prop:basic} a decoupling field \eqref{fbsde}. It is clear that if $g$ is Lipschitz, then by Proposition \ref{prop.calpha} the $u^{(k),\epsilon,i}$ are Lipschitz in space, uniformly in $k$ and $\epsilon$, from which it follows that $Du$ (and hence $v$) is bounded. 
\end{proof}

\begin{remark}
Let $(u,v)$ be the decoupling field produced by the above compactness argument. The convergence we obtain is strong enough to guarantee that $u$ is in fact a weak solution of the PDE \eqref{pde} in the sense of integration by parts, see e.g. Definition 4.1 in \cite{FENG2018959}. Verifying that \textit{any} decoupling field of \eqref{fbsde} corresponds to a weak solution of \eqref{pde} and vice-versa is much more subtle, and relates to a line of research on the connection between BSDEs and weak solutions of PDEs (rather than viscosity solutions) that dates back to \cite{barleslesigne}.
\end{remark}

\section{Application to stochastic differential games} \label{sec:games}

\subsection{Set-up and definition of Markovian Nash equilibrium} \label{subsec:gamedef}
We consider a game in which players $i = 1,...,n$ choose controls $\alpha^1,...,\alpha^n$ which take values in measurable sets $A^i \subset \R^k$, and influence the $d$-dimensional state process $X$ through the dynamics 
\begin{align*}
    dX_t = b(t,X_t, \vec{\alpha}(t,X_t)) ds + \sigma(t,X_t) \dot dB_t. 
\end{align*}
Here $\vec{\alpha}$ denotes $(\alpha^1,...,\alpha^k)$. The goal of player $i$ is to maximize the payoff functional 
 $J^i(\vec{\alpha}) = \bE[g^i(X_T) + \int_0^T r^i(t,X_t, \vec{\alpha}(t,X_t))ds]$. More precisely, the game is specified by the following data:
\begin{itemize}
    \item for each $i$, a number $k_i \in \N$ and a set $A^i \subset \R^{k_i}$ which represents the set of possible actions of player $i$ (we could take $A^i$ to be an arbitrary metric space, but we will use subsets of Euclidean space for simplicity of notation),
    \item a measurable function $b : [0,T] \times \R^d \times A \to \R^d$, where we set $A = \prod_{i = 1}^n A^i$,
    \item a measurable function $\sigma : [0,T] \times \R^d \to \R^{d \times d}$
    \item for each $i$, a measurable function $r^i : [0,T] \times \R^{d} \times A \to \R$, \item for each $i$, a measurable function $g^i : \R^d \to \R$. 
\end{itemize}
We define $\sA_i$ to be the set of \textit{bounded} measurable functions $[0,T] \times \R^d \to A^i$, and $\sA = \prod_{i = 1}^n \sA^i$. We assume for the moment that we have for each $t \in [0,T]$ and $x \in \R^d$ a unique strong solution to the SDE
\begin{align} \label{stateeqn}
    dX^{t,x}_s = b(s,X^{t,x}_s, \vec{\alpha}(s,X^{t,x}_s)) ds + \sigma(s,X^{t,x}_s)  dB_s, \, \, X^{t,x}_t = x.
\end{align}
For each $(t,x) \in [0,T] \times \R^d$, player $i$ has a payoff functional $J^i_{t,x} : \sA \to \R$, defined by
\begin{align*}
    J^i_{t,x}(\vec{\alpha}) = \bE[g^i(X_T^{t,x}) + \int_t^T r^i(s,X^{t,x}_s, \vec{\alpha}(s,X^{t,x}_s))ds]. 
\end{align*}
We also assume for the moment that the integrals appearing in the definition of $J^i_{t,x}$ are well-defined for each $\vec{\alpha} \in \sA$.
\begin{definition}
We say that $\vec{\alpha} = (\alpha^1,...,\alpha^n) \in \sA$ is a Markovian Nash equilibrium (MNE) for the game with data $(A^i,b,\sigma, r,g)$ if for each $i \in \{1,...,n\}$, $\beta \in \sA^i$ and each $(t,x) \in [0,T] \times \R^d$, we have
\begin{align*}
    J^i_{t,x}(\vec{\alpha}) \geq J^i_{t,x}(\vec{\alpha}^{-i},\beta), 
\end{align*}
where $(\vec{\alpha}^{-i},\beta) \coloneqq (\alpha^1,...\alpha^{i-1},\beta,\alpha^{i+1},...,\alpha^n) \in \sA$. 
\end{definition}
Our approach to producing Nash equilibria will be through an appropriate FBSDE system, which we describe here. We define for each $i$ the (reduced) Hamiltonian $H^i : [0,T] \times \R^d  \times \R^d \times A \to \R$ by
\begin{align*}
    H^i(t,x,p^i,a^1,...,a^n) = b(t,x,a^1,...,a^n) \cdot p^i + r^i(t,x,a^1,...,a^n). 
\end{align*}
We assume that the \textbf{generalized Isaacs condition} holds, i.e. there exist measurable functions $\hat{a}^i : [0,T] \times \R^d \times (\R^d)^n \to A^i$ such that for each $x,p \in \R^d$, 
\begin{align} \label{isaacs}
    H^i(t,x,p^i,\hat{a}(t,x,p)) = \sup_{a \in A^i} H^i(t,x,p^i,(\hat{a}^{-i}(t,x,p),a)),
\end{align}
where we write $p = (p^1,...,p^n) \in (\R^d)^n$, $\hat{a}(t,x,p) = (\hat{a}^1,...,\hat{a}^n)(t,x,p)$ and $(\hat{a}^{-i},a) = (\hat{a}^1,...,\hat{a}^{i-1}, a, \hat{a}^{i+1},...,\hat{a}^n)$. Then, we pose the following FBSDE
\begin{align} \label{valuefbsde}
    \begin{cases}
    dX_t = b(t,X_t, \hat{a}(t,X_t,  Z_t \sigma^{-1}(t,X_t)))dt + \sigma(t,X_t) dB_t, \\
    dY_t = - r(t,X_t, \hat{a}(t,X_t, Z_t\sigma^{-1}(t,X_t))) dt + Z_t  dB_t, \, \, Y_T = g(X_T), 
    \end{cases}
\end{align}
along with the HJB PDE system
\begin{align} \label{hjbsystemgen}
    \partial_t u^{i} +  \tr(a D^2 u^i) + H^i(t,x,D u^i, \hat{a}(t,x,Du)) = 0, \, \, u(T,x) = g(x).
\end{align}
Because of the connection between the HJB system \eqref{hjbsystemgen} and the FBSDE \eqref{valuefbsde}, and the well known connection between \eqref{hjbsystemgen} and Markovian Nash equilibria (see e.g. Section 6.3 of \cite{carmona2018probabilistic}), we expect that if $(u,v)$ is a decoupling field of \eqref{valuefbsde}, then $\alpha^*(t,x) \coloneqq \hat{a}(t,x,\sigma^{-1}(t,x)v(t,x))$ is a MNE for the game. In particular, if $u$ is a classical solution to \eqref{pde}, then we expect that $\hat{a}(t,x) \coloneqq \hat{a}(t,x,Du(t,x))$ is a MNE. To make this precise, we impose some mild conditions on the data. 

\be \label{hyp.g} \tag{$H_G$}
\begin{cases}
\text{The generalized Isaacs condition holds with optimizer $\hat{a}$}, \\ \text{the map $\sigma$ satisfies the conditions appearing in \ref{hyp.sigmab}, $g$ is bounded and the estimates} \\
\hspace{.5cm} 1) \,\, |r(t,x,a)| \leq C_G(1 + |a|^2), \\ 
\hspace{.5cm} 2) \,\, |\hat{a}(t,x,p)| \leq C_G(1 + |p|)
\\
\hspace{.5cm} 3) \,\, |b(t,x,a)| \leq C_G(1 + |a|) \\
\text{hold for all } (t,x,a) \in [0,T] \times \R^d \times A, \,\, p \in (\R^d)^n.
\end{cases}
\ee

The following is a verification result, stated in terms of the FBSDE \eqref{valuefbsde} instead of the PDE \eqref{hjbsystemgen}.  

\begin{proposition} \label{pro:fbsdeclne}
Suppose that \ref{hyp.g} holds, and that \eqref{valuefbsde} has a decoupling field $(u,v)$ with $v$ bounded. Then $\vec{\alpha}(t,x) \coloneqq \hat{a}(t,x,v(t,x)\sigma^{-1}(t,x))$ is a MNE for the game with data $(A^i,b,\sigma,r,g)$.
\end{proposition}

\begin{proof} We will show that $\vec{\alpha}$ is a closed loop Nash equilibrium in three steps. \newline 
\textit{Step 1:} We first establish that $u^i(t,x) = J^i_{t,x}(\vec{\alpha})$. Indeed, notice that if $X$ solves 
\begin{align*}
    d\tilde{X}_s = b(s,\tilde{X}_s, (\vec{\alpha}(s,\tilde{X}_s)) ds + \sigma(s,\tilde{X}_s)  dB_s, \, \, \tilde{X}_t = x
\end{align*}
on $[t,T]$, and $(\tilde{Y},\tilde{Z}) = (u(\cdot, \tilde{X}),v(\cdot, \tilde{X}))$, then we have  
\begin{align*}
    \tilde{Y}_s = g(\tilde{X}_T) + \int_s^T r(r,\tilde{X}_r, \vec{\alpha}(r,\tilde{X}_r)) dr - \int_s^T \tilde{Z_r}  dB_r, 
\end{align*}
and in particular $u^i(t,x) = \tilde{Y}^i_t = J^i_{t,x}(\vec{\alpha})$. 
\newline
\textit{Step 2:} Fix $(t,x) \in [0,T] \times \R^d$, and choose $\beta \in A$ such that $(\vec{\alpha}^{-i}, \beta) \in \sA$. The second step is to construct a BSDE representation of $J_{t,x}(\vec{\alpha}^{-i},\beta)$. Denote by $X$ the solution on $[t,T]$ to the equation $dX_s = b(s,X_s, (\vec{\alpha}^{-i}, \beta)(s,X_s)) ds + \sigma(s,X_s)  dB_s$ with initial condition $X_t = x$.
We now introduce the BSDE 
\begin{align} \label{comparebsde}
    Y'_s = g(X_T) + \int_s^T r(u,X_u, (\vec{\alpha}^{-i},\beta)(u,X_u)) du - \int_{s}^T Z'_u dB_u. 
\end{align}
Under \ref{hyp.g}, $r(\cdot, X, (\vec{\alpha}^{-i}, \beta)(\cdot,X)) \in L^2([0,T] \times \Omega)$, so \eqref{comparebsde} has a unique solution $(Y',Z')$, which clearly satisfies $Y^{'i}_t = J^i_{t,x}(\vec{\alpha}^{-i}, \beta)$. 
\newline
\textit{Step 3:} Having established the identities $J_{t,x}^i(\vec{\alpha}) = u^i(t,x)$, $J_{t,x}^i(\vec{\alpha}^{-i},\beta) = Y^{'i}_t$, we now complete the proof by showing that 
$
    u^i(t,x) \geq Y'_t. 
$
To do this, we define $Y = u(\cdot, X)$, $Z = v(\cdot, X)$. Under \ref{hyp.g}, we see that we can write 
$
    dX_s = \sigma(s,X_s) d\tilde{B}_s, 
$
where 
$
    \tilde{B} = B - \int b(\cdot, X, (\vec{\alpha}^{-i}, \beta)(\cdot, X) ) \sigma^{-1}(\cdot, X) ds
$
and $\tilde{B}$ is a Brownian motion under an equivalent probability measure. 
By virtue of the fact that $(u,v)$ is a decoupling field for \eqref{valuefbsde}, we get that (following the computations in the proof of Proposition \ref{prop:basic}, and writing $\hat{a}$ as a shortcut for $\hat{a}(\cdot,X, Z \sigma^{-1}(\cdot,X))$ for brevity),
\begin{align*}
    Y_s^i
    &= g^i(X_T) + \int_s^T  H^i(u,X_u,\sigma^{-1}(u,X_u) Z_u^i , \hat{a}) dr - \int_{s}^T Z_u^i \cdot \big( \sigma^{-1}(u,X_u) dX_u \big)\\
    &= g^i(X_T) + \int_s^T \Big(H^i(u,X_u, \sigma^{-1}(u,X_u)Z_u^i , \hat{a}) \\ &- (\sigma^{-1}(u,X_u)Z^i_u) \cdot b(u,X_u, (\vec{\alpha}^{-i}, \beta)(u,X_u))\Big) ds  - \int_s^T Z^i_u dB_u. 
\end{align*}
Thus, setting $\Delta Y = Y - Y'$, $\Delta Z = Z - Z'$, we have 
\begin{align*}
    \Delta Y_s^i = \int_s^T \big(H^i(u,X_u, \sigma^{-1}(u,X_u)Z_u^i , \hat{a}) - H^i(u,X_u,\sigma^{-1}(u,X_u)Z_u^i,(\vec{\alpha}^{-i},\beta)(u,X_u)\big)dr - \int_s^T \Delta Z_u^i dB_u
\end{align*}
Since $
    H^i(u,X_u, \sigma^{-1}(u,X_u) Z_u^i , \hat{a}) - H^i(u,X_u,\sigma^{-1}(u,X_u)Z_r^i,(\vec{\alpha}^{-i},\beta)(u,X_u)) \geq 0$, 
we conclude that $J^i_{t,x}(\vec{\alpha}) = u^i(t,x) =  Y^i_t \geq Y^{'i}_t = J^i_{t,x}(\vec{\alpha}^{-i},\beta)$.
\end{proof}

\subsection{Games with diagonal cost structures and additive drift}

We now describe a general class of games to which our results on FBSDEs can be applied. We assume that the dynamics take the form 
\begin{align*}
    dX_t = \big( \sum_{j = 1}^n b^j(t,X_t,\alpha_t^j) \big) dt + \sigma(t,X_t) dB_t
\end{align*}
while the payoff for player $i$ takes the form 
\begin{align*}
 J^i_{t,x}(\vec{\alpha}) = \bE[g^i(X_T^{t,x}) + \int_t^T r^i (s,X^{t,x}_s,\alpha^i(s,X_s^{t,x})) dt].
\end{align*}
Player $i$'s Hamiltonian in this case is given by
\begin{align*}
    H^i(t,x,p^i,a^1,...,a^n) = \big(\sum_j b^j(t,x,a^j) \big) \cdot p^i + r^i(t,x,a^i). 
\end{align*}
In particular, the Isaacs condition holds as soon as there exists for each $i$ a measurable map $\hat{a}^i = \hat{a}^i(t,x,p^i) : [0,T] \times \R^d \times \R^d \to A^i$ such that 
\begin{align} \label{aidef}
    b^i(t,x,\hat{a}^i(t,x,p^i)) \cdot p^i + r^i(t,x,a^i) = \sup_a \big( b^i(t,x,\hat{a}^i(t,x,p^i)) \cdot p^i + r^i(t,x,a^i)\big)
\end{align}
for each $(t,x,a^i)$. 
We note that in terms of the notation introduced in the previous subsection, we have $b(t,x,a) = \sum_j b^j(t,x,a^j)$, $r^i(t,x,a) = r^i(t,x,a^j)$. 
Let us list the necessary assumptions on the data.
\be 
\begin{cases} \label{hyp.diag} \tag{$H_{\text{diag}}$}
    \text{The functions $b^i, \sigma, r^i, g^i$ are all continuous, $\sigma$ satisfies the conditions in \ref{hyp.sigmab}} \\
    \text{and there is a constant $\cdiag$ such that the estimates }\\
    \hspace{.5cm} 1) |b^i(t,x,a^i)| \leq \cdiag(1 + |a^i|) \\
    \hspace{.5cm} 2)  |g^i(x)| \leq \cdiag, \quad |g^i(x) - g^i(x')| \leq \cdiag |x - x'| \\
    \hspace{.5cm} 3) |r^i(t,x,a^i)| \leq \cdiag (1+ |a^i|^2) \\
    \text{hold for all $x,x' \in \R^d$, $t \in [0,T]$, $a^i \in A^i$. Moreover there exist continuous functions} \\ \text{$\hat{a}^i$ satisfying \eqref{aidef}, \text{and such that}} \\
    \hspace{.5cm} 4) |\hat{a}^i(t,x,p^i)| \leq \cdiag(1 + |p^i|). 
\end{cases}
\ee 

Note that if \ref{hyp.diag} holds, the FBSDE \eqref{valuefbsde} becomes  
\begin{align} \label{diagfbsde}
\begin{cases}
    dX_t = \big(\sum_j b^j(t,X_t,\hat{a}^j(t,X_t,\sigma^{-1}(t,X_t) Z^j )) \big) dt + \sigma(t,X_t)  dB_t, \\
    dY_t^i = -\big( r^i(t,X_t,\hat{a}^i(t,X_t,\sigma^{-1}(t,X_t)Z^i )) \big) dt + Z_t^i \cdot dB_t^i, \, \, Y_T = g(X_T). 
\end{cases}
\end{align}
\begin{theorem} \label{prop:clneexist}
Suppose that \ref{hyp.diag} holds. Then the FBSDE \eqref{diagfbsde} has a decoupling field $(u,v)$ with $v$ bounded. Consequently, $\vec{a}(t,x) = \hat{a}(t,x,\sigma^{-1}(t,x)v(t,x))$ is a MNE for the game with data $(A^i,b,\sigma, r,g)$.
\end{theorem}

\begin{remark}
It is natural to ask whether the equilibrium we produce is unique. If we only impose \ref{hyp.g}, we cannot expect uniqueness, in short because we cannot guarantee uniqueness of the FBSDE \eqref{diagfbsde} (or of the corresponding PDE) without additional regularity conditions. Nevertheless, under appropriate technical conditions one can guarantee a one-to-one correspondence between Markovan Nash equilibria and certain generalized solutions of the HJB system by following the arguments in Proposition 6.27 in \cite{carmona2018probabilistic}. This gives one way to check that if $(u,v)$ is a decoupling field for \eqref{diagfbsde} with $v$ bounded, then $u$ must in fact solve the corresponding PDE in an appropriate sense. To make this rigorous requires a discussion of weak solutions for the PDE system \eqref{pde}, regularity properties of scalar Hamilton-Jacobi equations with irregular Hamiltonians and the It\^o Krylov formula. We do not pursue this analysis for the sake of brevity. 
\end{remark}

\begin{proof}
This is a matter of checking that if \ref{hyp.diag} holds, then the functions $b, \sigma, f, g$ with 
\begin{align*}
    b(t,x,z) = \sum_j b^j(t,x,\hat{a}^j(t,x, \sigma^{-1}(t,x) z^j), \quad 
    f^i(t,x,z) = r^i(t,x, \hat{a}^i(t,x, \sigma^{-1}(t,x) z^i))
\end{align*}
satisfy the conditions of Theorem \ref{thm.exist}. The only thing which is not obvious is \ref{hyp.ab}. For this, we note that we can easily check $
     |f^i(t,x,z)| \leqc 1 + |z^i|^2$,
 which implies that the condition (AB) is satisfied, with $\{a_m\} = \{\pm \lambda e_m\}_{m = 1}^n$, $\rho = \lambda$, where $\lambda$ is a large enough positive constant and $e_m$ is the $m^{th}$ standard basis vector in $\R^n$. 
\end{proof}

\bibliographystyle{amsalpha}
\bibliography{fbsde}

\end{document}